\documentclass[11pt,a4paper,reqno]{amsart}
\usepackage{graphicx}

\usepackage{amsmath} 
\usepackage{amssymb}  
\usepackage{graphics}
\usepackage{multirow}
\usepackage{epsfig}
\usepackage[ruled,vlined]{algorithm2e}
\usepackage{amsfonts}
\usepackage{mathrsfs}
\usepackage{bbm}
\usepackage{color}
\usepackage{pgfpages}
\usepackage{tikz}
\usetikzlibrary{arrows,snakes,backgrounds}

\UseRawInputEncoding
\usepackage{epsfig}
\usepackage{cite}
 \newtheorem{thm}{Theorem}[subsection]

 \newtheorem{proposition}[thm]{Proposition}
 \theoremstyle{definition}
 
 \theoremstyle{remark}
 \newtheorem{rem}[thm]{Remark}

\begin{document}
\setcounter{page}{1}
\bigskip
\bigskip
\title[R. D'Ambrosio \and C. Scalone: On the numerical solution of stochastic oscillators ] {On the numerical solution of stochastic oscillators driven by time-varying and random forces}
\author[]{R. D'Ambrosio$^1$ \and C. Scalone$^1$
 }
\thanks{$^1$ Department of Engineering and Computer Science and Mathematics, University of L'Aquila \\
               Via Vetoio, Loc. Coppito - 67100 L'Aquila, Italy,
\\ \indent\,\,\,e-mail: raffaele.dambrosio@univaq.it, carmela.scalone@univaq.it
\\ \indent
  \em \,\,\,Manuscript received xx}

\begin{abstract}
 In this work, we provide a specific trigonometric stochastic numerical method for linear oscillators with high constant frequencies, driven by a nonlinear time-varying force and a random force. We present some theoretical considerations and numerical experiments on popular related physical models.\\
\bigskip

\\


\end{abstract}
\maketitle

\smallskip
\section{Introduction}
Providing accurate numerical methods for discretizing mathematical models of oscillating phenomena without imposing severe stepsize restrictions is a very thorough topic. Scientific literature is extremely rich in models describing the dynamics of different types of oscillators, both in deterministic and stochastic setting. Because of the variety of models, an appropriate numerical treatment for any particular kind of oscillator is required. More specifically, many interesting examples of stochastic oscillator are obtained by the introduction in the equation of a deterministic oscillator of a \textit{noisy ingredient}, which can be, for example, an additive and/ or a multiplicative noise, a random  frequency, a random damping, and so on (see \cite{GittBook} and references therein for a survey).

In this paper, we focus on stochastic harmonic undamped oscillators driven by both a deterministic time-dependent force and a random Gaussian forcing.
This is modeled by a second order stochastic differential equation of the type:
\begin{equation}
\label{oscillator}
\ddot{x}=-\omega^2 x +g(t)+\varepsilon \xi(t)
\end{equation}
where $\xi(t)$ is a white noise process satisfying $\mathbb{E}\vert \xi(t)\xi(t') \vert =\delta(t-t')$ and $\omega$ is a positive real constant.

This kind of stochastic oscillators is very popular in physics literature, see \cite{GittBook,gitterman2013,lnp}, therefore a specific numerical treatment may be very useful for approximating their solutions. In the existing literature, there are several works devoted to investigate different aspects of stochastic oscillatory problems.
An analysis of long-term features of the linear oscillator 
$$\ddot{x}=-x +\varepsilon \xi(t)$$
(i.e. equation \eqref{oscillator} with $\omega=1$ and  $g\equiv 0$) is provided in \cite{smh04}.
In \cite{Cohen,CohenSigg}, the authors introduced a family of stochastic trigonometric numerical methods based on a variation of constant formula for nonlinear equations of the form
\begin{equation}
\label{oscill_cohen}
\ddot{x}=-\omega^2 x + g(x)+\varepsilon \xi(t)\end{equation}
where $g$ is a positional force and it is supposed to be potential.
The investigation of long-term properties of the equation, of the numerical approximation and their comparison, is typical in the study of stochastic oscillatory problems (see, for instance, \cite{BLL,bl09,citro,Cohen,CohenSigg,ds,cruz,senosian,vilmart} and references therein). 
Our aim is to join two ingredients, i.e. a variation of constant formula and specific quadrature rules to define a numerical treatment specific for problem \eqref{oscillator}.
The paper is organized as follows: in Section \ref{var_const}, we start from a semi-discretization based on a variation of constants formula; in Section \ref{filon} we remember the essential aspects of Filon's quadrature formulas. 
Section \ref{convergence} is dedicated to  theoretical convergence results and, finally, in Section \ref{tests} we test our method on famous physical models.

\section{Semidiscretization based on a variation of constant.}
\label{var_const}
Equation \eqref{oscillator} is equivalent to the following first order system of two equations in the variables $X_t$ (the position of the particle) and $V_t$ (its velocity):
\begin{equation}
\label{syst_XV}
\begin{cases}
dX_t = V_t dt\\
dV_t = -\omega^2 X_t dt+ g(t) dt+\varepsilon dW_t
\end{cases}
\end{equation}

Let us consider the matrix

$$R(t \omega)=\exp \left( \ t \begin{bmatrix} 0 & I \\
-\omega^2 & 0
\end{bmatrix} \right) = \begin{bmatrix}
\cos(t \omega)  &  \omega^{-1} \sin(t \omega)\\
- \omega \sin(t \omega)  &   \cos(t \omega)
\end{bmatrix}. $$

Then, the variation of constant formula applied to the system \eqref{syst_XV} is given by
\begin{equation}
\label{var_constant}
\begin{aligned}
\begin{pmatrix}
X_t \\
V_t
\end{pmatrix}= R(t \omega) \begin{pmatrix}
X_0 \\
V_0
\end{pmatrix} &+ \int_0^t  \begin{pmatrix}
\omega^{-1} \sin((t-s) \omega) \\ \cos( (t-s) \omega)
\end{pmatrix} g(s) ds \\
&+ \varepsilon \int_0^t  \begin{pmatrix}
\omega^{-1} \sin((t-s) \omega) \\ \cos( (t-s) \omega)
\end{pmatrix} dW_s.
\end{aligned}
\end{equation}
The discretization of the stochastic part leads to
\begin{equation}
\label{semidiscr}
\begin{aligned}
\begin{pmatrix}
X_{n+1} \\
V_{n+1} 
\end{pmatrix}= R(h \omega) \begin{pmatrix}
X_{n}  \\
V_{n} 
\end{pmatrix} &+ \int_{t_n}^{t_{n+1} } \begin{pmatrix}
\omega^{-1} \sin((t_{n+1} -s) \omega) \\ \cos( (t_{n+1} -s) \omega)
\end{pmatrix} g(s) ds \\
&+ \varepsilon  \begin{pmatrix}
\omega^{-1} \sin(h \omega) \\ \cos( h \omega)
\end{pmatrix} \Delta W_n,
\end{aligned}
\end{equation}
where $\Delta W_n= W(t_{n+1})-W(t_n)$ denotes the Wiener increment.
As nicely discussed in \cite{CohenSigg}, when $\omega \gg 1$, the problem is stiff, therefore numerical methods able to get good accuracy with step sizes $h$, whose product with $\omega$ do not need to be small.
This means that the error bounds in the position of the methods should be independent of the product $h \omega$.
The use of the variation of constant formula \eqref{var_constant} to design numerical schemes suited for such problems, is very used in the deterministic setting, see \cite{hlw,HockOster}.
The methods derived in \cite{Cohen} and \cite{CohenSigg} are the stochastic counterpart of the trigonometric methods.
Our problem \eqref{oscillator} differs from \eqref{oscill_cohen} because we have the deterministic force $g=g(t)$ acting on the system and we aim to construct accurate numerical procedures combining the variation of constant formula \eqref{var_constant} and appropriare quadrature rules.

\section{Filon's quadrature formula}
\label{filon}
In this section, we aim to remind few relevant issues on Filon quadrature formulas, see \cite{DavisRabinowitz,Filon}, which are perfectly suitable for our case. Filon formulas are specialized for integrals of the type
\begin{equation}
\label{integrale_filon}
\int_a^b \sin(k x) \psi(x) dx,
\end{equation}
where $k$ is a real constant. 
When $k$ is large, the rapid oscillation of $\sin(kx)$, the ordinary quadrature formulae require the division of the range of integration into a wide number subinterval.
We consider even number of nodes $2n$ in the interval $[a,b]$, so that
$$b=a+2nh$$
and we set
$$\theta=kh, \quad x_r=a+rh,\quad \psi_r=\psi(x_r).$$
Filon formula is then given by
\begin{equation}
\label{formula_filon}
\int_a^b \sin(k x) \psi(x) dx=h [ \alpha( \psi(a) \cos(k a)-\psi(b)\cos(k b)+ \beta S_{2r}+ \gamma S_{2r-1}],
\end{equation}
where
$$
\begin{aligned}
\alpha&= 1/\theta+ \cos \theta \sin \theta/ \theta^2-2 \sin^2(\theta)/\theta^3,\\
\beta&= 2[(1+cos^2\theta)/\theta^2-2 \cos \theta \sin \theta/ \theta^3,\\
\gamma&=4 (\sin \theta/\theta^3-\cos \theta/\theta^2),\\
S_{2r}&= 2 \sum_{r=0}^n \psi(x_{2r}) \sin(x_{2r})-\psi(a)\sin(k a)-\psi(b) \sin(k b),\\
 S_{2r+1}&= \sum_{r=1}^n \psi(x_{2r-1}) \sin(x_{2r-1}).
\end{aligned}
$$
For our purposes, it is appropriate to recall the following associated formula 
\begin{equation}
\int_a^b \sin(k x+ z) \psi(x) dx=h [ \alpha( \psi(a) \cos(k a+z)-\psi(b)\cos(k b+z)+ \beta S_{2r}+ \gamma S_{2r-1}]
\end{equation}
setting $z=\frac{\pi}{2}$, we get the formula for $\int_a^b \cos(k x) \psi(x) dx$.
Let $E_s, E_c$ designate respectively the error in Filon's sine and cosine formulas.
Of we require to choose $h$ in a way that $\theta <1$
this means that for large values of $k$, the choice of a small $h$ is mandatory.
However, if $f(x)$ can be well approximated by a piecewise quadratic function $p(x)$ using a coarse mesh, then we do not need for $\theta$ to be small. More precisely, if 
$$\max_{x \in [a,b]}\vert f(x) -p(x) \vert \leq \delta, $$
then, a uniform bound on the error in Filon's formula is given by $(b-a) \delta$.
When
$\theta <1$, set 
$$H(\theta)= \left\lvert \frac{\sin \theta}{3 \theta^2}+\frac{\sin \theta}{ \theta^3}-\frac{\cos \theta}{ \theta^4}\right\rvert $$ and $$M=\max_{a \leq x \leq b} \vert f^{(3)}(x) \vert,$$
then it has been shown in \cite{DavisRabinowitz} that
\begin{equation}
\label{filon_error}
E_s, E_c \leq H(\theta) M (b-a) h^3 +\mathcal{O}(h^4).
\end{equation}

\section{Convergence}
\label{convergence}
This section is dedicated to study the convergence of the introduced scheme.
We start recalling the following result form \cite{CohenSigg} concerning the the global mean-square error of the stochastic
trigonometric
 integrator  for the linear case.
\begin{proposition}
\label{prop1}
Consider the numerical solution of the linear problem 
$$ \ddot{x}= -\omega^2 x+ \varepsilon \dot{W}_t$$
given by 
\begin{equation}
\begin{pmatrix}
X_{n+1} \\
V_{n+1} 
\end{pmatrix}= R(h \omega) \begin{pmatrix}
X_{n}  \\
V_{n} 
\end{pmatrix} + \varepsilon  \begin{pmatrix}
\omega^{-1} \sin(h \omega) \\ \cos( h \omega)
\end{pmatrix} \Delta W_n.
\end{equation}
The mean-square errors of the above scheme, with a step size $h \leq h_0 $ (with a sufficiently small $h_0$ independent of $\varepsilon$) for
which $h/ \varepsilon \geq c_0 > 0$ holds,
satisfy
$$ \left(  \mathbb{E}\vert X_{t_n} - X_n \vert^2 \right) ^{1/2}\leq C \varepsilon\leq Ch,$$
$$ \left(  \mathbb{E}\vert V_{t_n} - V_n \vert^2 \right) ^{1/2}\leq C T^{1/2},$$
for $nh \leq T$.
The constant $C$ is independent of $\varepsilon$, $h$ and $n$.
\end{proposition}
Following the approach of \cite{CohenSigg} we prove the following result.
\begin{thm}
Let us suppose  the assumptions of Proposition \ref{prop1} satisfied. Let us suppose that on each subinterval $[t_{n-j-1},t_{n-j}]$ of the mesh, the function $g$ can be well approximated by a piecewise quadratic functions $p_j(x)$, i.e.
\begin{equation}
\max_{t \in [t_{n-j-1},t_{n-j}]} \vert g(t)-p_j(t) \vert \leq \delta_j.
\end{equation}
If $\delta := \max_j \delta_j$, we have that the global mean-square error in the position is given by
\begin{equation}
\mathbb{E}\vert X_n - X_{t_n} \vert^2 \leq  C h^2 +D \omega^{-2} h^4,
\end{equation}
where the constant $D$ depends only by  $T$ and $\delta$.
Instead, the global mean-square error in the velocity is estimated as
\begin{equation}
\mathbb{E}\vert V_n - V_{t_n} \vert^2 \leq ( C  + + \delta^2 h^2 ) T^2.
\end{equation}

\end{thm}
\begin{proof}
Let us write the numerical approximation as 
$$ X_n=  \cos(n h \omega)X_0+\omega^{-1} \sin (n h \omega) V_0 +d_1+s_1,$$
$$V_n= - \omega \sin(n h \omega) X_0+ \cos (n h \omega) V_0+ d_2+s_2, $$
with
$$\begin{aligned}
d_1&=\omega^{-1}\varepsilon  \sum_{j=0}^{n-1} \Bigg(  \cos(j h \omega) \mathcal{Q}^S_F(t_{n-j-1},t_{n-j},\omega)
 +\sin(j h \omega) \mathcal{Q}^C_F(t_{n-j-1},t_{n-j}, \omega) \Bigg),\\
s_1&=\omega^{-1}\varepsilon  \sum_{j=0}^{n-1}\left( \cos(j h \omega) \sin(j h \omega)+\sin( h \omega) \cos(j h \omega)\right)  \Delta W_{n-j-1} ,\\
d_2&=  \varepsilon  \sum_{j=0}^{n-1}\Bigg(  \sin(j h \omega) \mathcal{Q}^S_F(t_{n-j-1},t_{n-j}, \omega) -\cos(j h \omega) \mathcal{Q}^C_F(t_{n-j-1},t_{n-j}, \omega)\Bigg),\\
s_2&=\varepsilon  \sum_{j=0}^{n-1}\left( \sin(j h \omega) \sin( h \omega)-\cos( h \omega) \cos(j h \omega)\right)  \Delta W_{n-j-1}.
\end{aligned}$$
Let us also write the exact solution as 
$$ X_{t_n}=  \cos(n h \omega)X_0+\omega^{-1} \sin (n h \omega) V_0 +\hat{d}_1+\hat{s}_1,$$
$$V_{t_n}= - \omega \sin(n h \omega) X_0+ \cos (n h \omega) V_0+ \hat{d}_2+\hat{s}_2,$$
where \
$$\begin{aligned}
\hat{d}_1&=  \omega^{-1}\varepsilon  \sum_{j=0}^{n-1} \Bigg(  \cos(j h \omega) \int_{t_{n-j-1}}^{t_{n-j}}\sin( (t_{n-j}-s) \omega)g(s) ds\\
&\qquad\qquad\qquad+\sin(j h \omega) \int_{t_{n-j-1}}^{t_{n-j}}\cos( (t_{n-j}-s) \omega)g(s) ds    \Bigg),\\
\hat{s}_1 &= \omega^{-1}\varepsilon  \sum_{j=0}^{n-1}\Bigg(  \cos(j h \omega) \int_{t_{n-j-1}}^{t_{n-j}}\sin( (t_{n-j}-s) \omega) dW_s \\
&\qquad\qquad\qquad+\sin(j h \omega) \int_{t_{n-j-1}}^{t_{n-j}}\cos( (t_{n-j}-s) \omega) dW_s \Bigg),\\
\hat{d}_2&= \varepsilon  \sum_{j=0}^{n-1}\Bigg(  \sin(j h \omega) \int_{t_{n-j-1}}^{t_{n-j}}\sin( (t_{n-j}-s) \omega)g(s) ds\\
&\qquad\qquad\qquad-\cos(j h \omega) \int_{t_{n-j-1}}^{t_{n-j}}\cos( (t_{n-j}-s) \omega)g(s) ds \Bigg), \\
\hat{s}_2 &= \varepsilon  \sum_{j=0}^{n-1}\Bigg(  \sin(j h \omega) \int_{t_{n-j-1}}^{t_{n-j}}\sin( (t_{n-j}-s) \omega) dW_s\\
&\qquad\qquad\qquad+\cos( j h \omega) \int_{t_{n-j-1}}^{t_{n-j}}\cos( (t_{n-j}-s) \omega) dW_s \Bigg). 
\end{aligned}$$

The values $\mathcal{Q}^S_F(t_{n-j-1},t_{n-j}, \omega)$ and $\mathcal{Q}^C_F(t_{n-j-1},t_{n-j}, \omega)$ respectively denote the Filon's approximation of the integrals $\int_{t_{n-j-1}}^{t_{n-j}}\sin( (t_{n-j}-s) \omega)g(s) ds$    and $\int_{t_{n-j-1}}^{t_{n-j}}\cos( (t_{n-j}-s) \omega)g(s) ds$.   

The global mean-square error in the position is given by
\begin{equation*}
\mathbb{E}\vert X_n - X_{t_n} \vert^2= \mathbb{E}\vert d_1 + s_1 - \hat{d}_1-\hat{s}_1 \vert^2 \leq 2 \mathbb{E}\vert d_1-\hat{d}_1 \vert^2 +2 \mathbb{E}\vert s_1-\hat{s}_1 \vert^2.
\end{equation*}
The term $\mathbb{E}\vert d_1-\hat{d}_1 \vert^2$ is the global mean-square error in the linear case (see Prop. \ref{prop1} and \cite{CohenSigg}).
Since
\begin{align*}
\vert d_1-\hat{d}_1 \vert &\leq   \omega^{-1}\varepsilon  \sum_{j=0}^{n-1} \lvert  \mathcal{E}^S_F(t_{n-j-1},t_{n-j},\omega)
 +\mathcal{E}^C_F(t_{n-j-1},t_{n-j}, \omega) \rvert   \\
 & \leq\omega^{-1}\varepsilon  \sum_{j=0}^{n-1} 2 \delta_j h \leq
    2 \omega^{-1}\varepsilon  \delta n h = 2 \omega^{-1}\varepsilon  \delta T,
\end{align*}
then
\begin{equation*}
\mathbb{E}\vert X_n - X_{t_n} \vert^2 \leq 4 C \varepsilon^2 + 8 \omega^{-2}\delta ^2 T^2 \varepsilon^2.
\end{equation*}

The global mean-square error in the velocity is given by
\begin{equation*}
\mathbb{E}\vert V_n - V_{t_n} \vert^2= \mathbb{E}\vert d_2 + s_2 - \hat{d}_2-\hat{s}_2 \vert^2 \leq 2 \mathbb{E}\vert d_2-\hat{d}_2 \vert^2 +2 \mathbb{E}\vert s_2-\hat{s}_2 \vert^2.
\end{equation*}
The term $\mathbb{E}\vert s_2-\hat{s}_2 \vert^2$ is the global mean-square error in the linear case (see Prop. \ref{prop1} and \cite{CohenSigg}) in the velocity. 

Since, as before, 
$$ \vert \hat{d}_2-\hat{s}_2 \vert \leq T \delta \varepsilon,$$
finally we get 
\begin{equation*}
\mathbb{E}\vert V_n - V_{t_n} \vert^2 \leq C T^2 \delta^2 + \varepsilon^2 T^2.
\end{equation*}
\qed
\end{proof}

\begin{rem}
In the case in which the ``piecewise quadratic" hypothesis on $g$ is not satisfied, so we have to choose $h$ in a way that $\theta <1$, we can use the bounds \eqref{filon_error} to  give an estimate of the mean square-errors. Since
\begin{align*}
\vert d_1-\hat{d}_1 \vert \leq &  \omega^{-1}\varepsilon  \sum_{j=0}^{n-1} \lvert  \mathcal{E}^S_F(t_{n-j-1},t_{n-j},\omega)
 +\mathcal{E}^C_F(t_{n-j-1},t_{n-j}, \omega) \rvert  \\
 &\leq \omega^{-1}\varepsilon  \sum_{j=0}^{n-1} 2 H(\theta) M h (h/r)^3 \leq 2 \omega^{-1}\varepsilon  n h  H(\theta) M (h/r)^3,
\end{align*}
we have
\begin{equation*}
\mathbb{E}\vert X_n - X_{t_n} \vert^2 \leq 4 C \varepsilon^2 + 8 \omega^{-2} T^2 H^2(\theta) M^2  h_1^6 \varepsilon^2,
\end{equation*}
where $h_1=h/r$ in the quadrature formula.
\end{rem}

\section{Numerical Experiments}
\label{tests}
We propose an example very popular in physics literature (see \cite{Bulsara,GittBook,gitterman2013})
\begin{equation}
\label{cosine_example}
\ddot{x}=-\omega^2 x -Q \cos(\Omega t) +\varepsilon \xi(t).
\end{equation}
We integrate \eqref{cosine_example} over the interval $[0,1]$, with initial conditions $X_0=0.8$ and $V_0=1$.
We employ the Filon's formula choosing five nodes. The values of the constants are $Q=5$, $\Omega=20$ and $\varepsilon=0.3$. Following the approach in \cite{H2001}, Figure \ref{b1000_eps03} provides the plot of the strong errors  of the numerical solutions obtained by \eqref{semidiscr}, discretizing the integrals by Filon, Lobatto (also with five points) and the trapezoidal quadrature formulas, at the endpoint $T=1$. More precisely, in Figure \ref{b1000_eps03}, we have set $\omega=10$ (top) and $\omega=50$ (bottom). In Figure \ref{b10000_eps03}, we plot the results for  $\omega=100$. These figures show that Filon method outperforms the other ones. 

\begin{figure}
\begin{tabular}{c}
         \includegraphics[width=\textwidth]{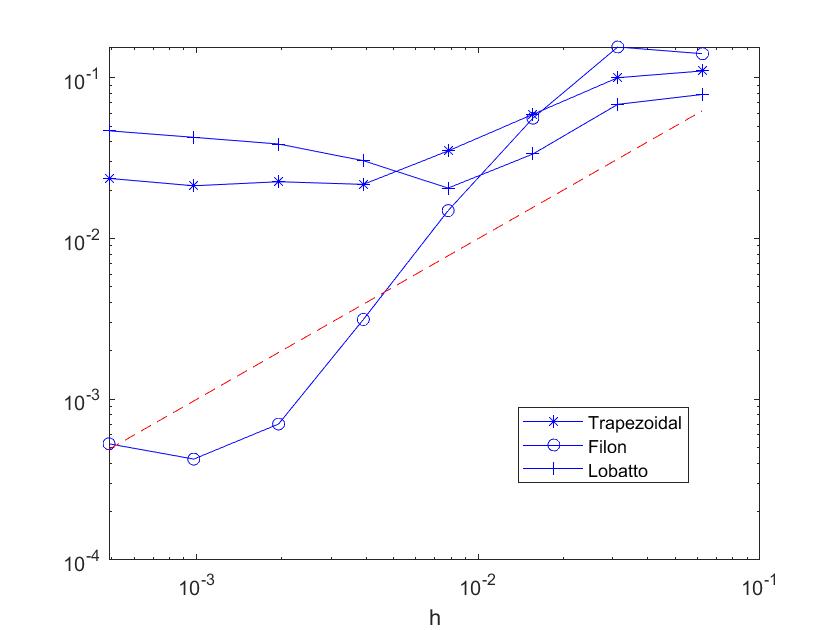}\\
         \includegraphics[width=\textwidth]{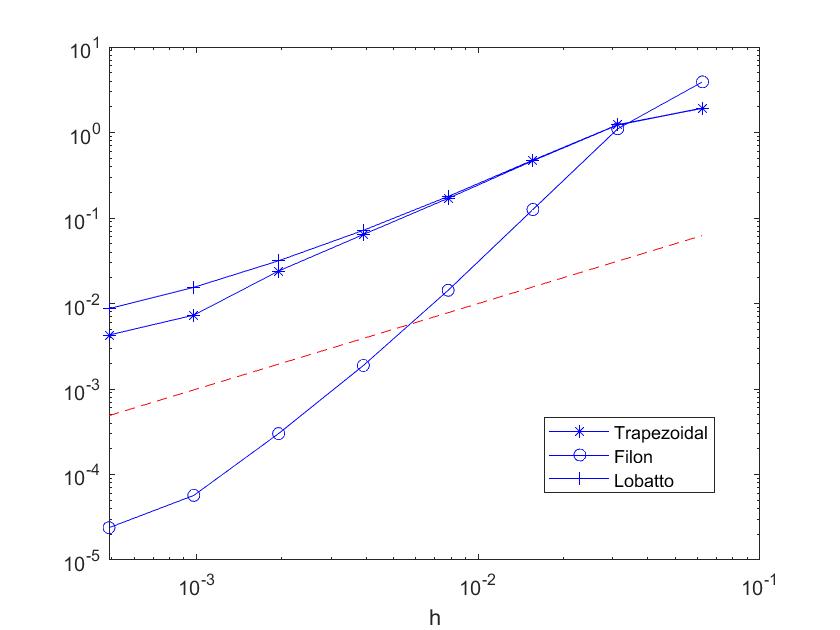}
\end{tabular}
            \caption{Strong errors at $T=1$ of the numerical solutions. In the top figure we set $\omega=10$, while in the bottom figure on the left $\omega=50$.}
        \label{b1000_eps03}
\end{figure}

\begin{figure}[h]
\centering
\includegraphics[width=\textwidth]{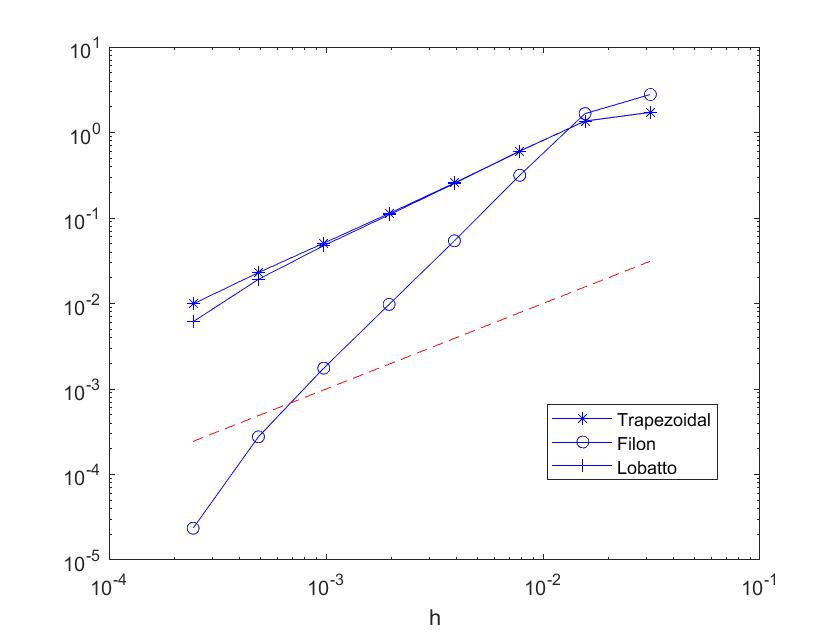}
\caption{Strong errors  of the numerical solutions at $T=1$ for  $\omega=100$.} \label{b10000_eps03}
\end{figure}

As second example, we propose the problem
\begin{equation}
\label{cosinesine_example}
\ddot{x}=-\omega^2 x -Q \cos(\Omega t)-R\sin(\Theta t) +\varepsilon \xi(t),
\end{equation}
with $\varepsilon=0.3$,  $Q=3$, $\Omega=30$, $R=2$ and $\Theta=25$.
We integrate \eqref{cosine_example} over the interval $[0,1]$, with initial conditions $X_0=0.8$ and $V_0=1$.  Figure \ref{sc_eps03} shows the strong errors in $T=1$, setting $\omega=50$ (top) and $\omega=150$ (bottom), of the three methods, corresponding to the three different quadrature to fully discretize \eqref{semidiscr}. Also in this case Filon method outperforms the other ones. 

\begin{figure}
\begin{tabular}{c}
         \includegraphics[width=\textwidth]{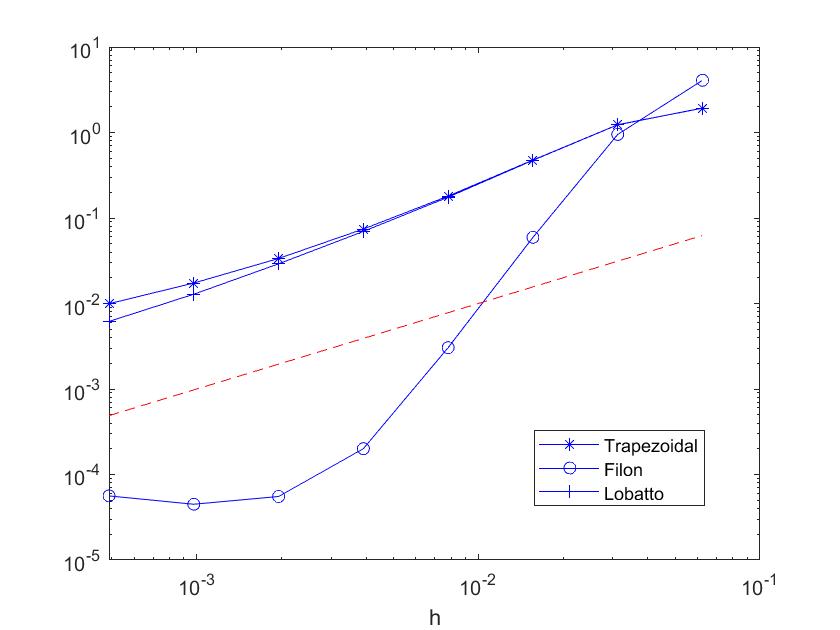}\\
         \includegraphics[width=\textwidth]{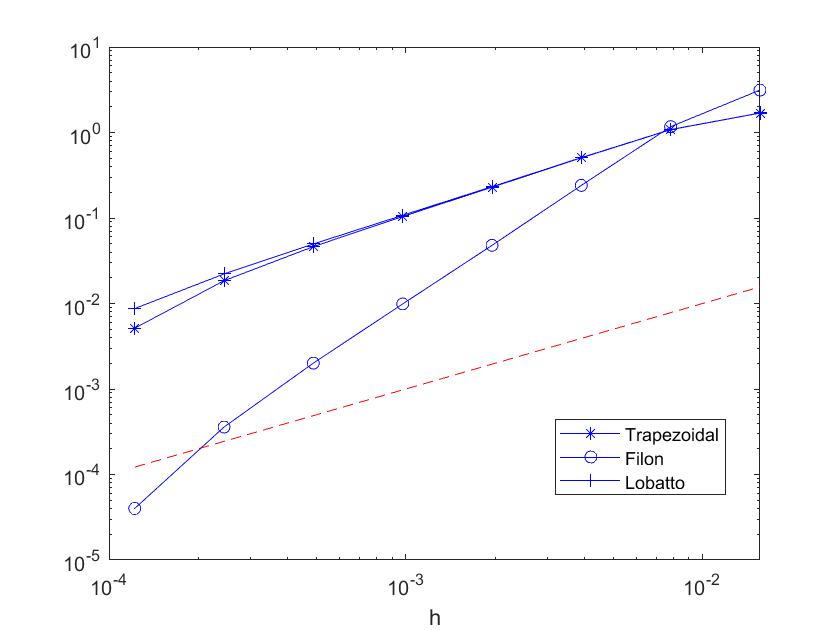}
\end{tabular}
            \caption{Strong errors in the position for $\omega=50 $ (top) and in for $\omega=150$ (bottom), related to problem \ref{cosinesine_example}.}
        \label{sc_eps03}
\end{figure}

\section{Conclusions}
This work is dedicated to present a specific numerical strategy for stochastic oscillators of type \eqref{oscillator}. Numerical experiments show that our approach give the best generalized performance against a full discretization of \eqref{semidiscr} via Lobatto or Trapezoidal rule. In particular, the Trapezoidal option may be seen as the analogue of the method of \cite{CohenSigg} for the case in which $g=g(t)$. Moreover, in many examples of oscillators we can find the co-presence of a forcing term $g=g(t)$ and a non linearity with respect to $x$. In these cases, it is possible to join the approaches to accurately simulate the solutions. Our general idea is to correctly adapt the numerical strategy to the problem; in particular, models like \eqref{oscillator} with time-varying frequencies and/or presence of damping may be objects of future works.
\bigskip

\noindent{\\ \bf Acknowledgments}\\
This work is supported by GNCS-INDAM project and by PRIN2017-MIUR project. The authors are member of the INDAM Research group GNCS. The authors are thankful to Dr. Monia Capanna for their helpful discussions.

\bigskip


\end{document}